\documentclass[12pt]{amsart}
\usepackage{color}
\usepackage[all]{xy}
\usepackage{amssymb,amsmath}

\usepackage{mathrsfs}
\usepackage{eucal}

\usepackage{graphicx}
\usepackage[font={small,it},labelfont=bf]{caption}

\usepackage{setspace}

\date{May 28, 2015}

\calclayout
\newtheorem{dummy}{anything}[section]
\newtheorem{theorem}[dummy]{Theorem}
\newtheorem*{thma}{Theorem A}
\newtheorem*{thmb}{Theorem B}
\newtheorem*{thmc}{Corollary C}
\newtheorem{lemma}[dummy]{Lemma}

\newtheorem{corollary}[dummy]{Corollary}

\theoremstyle{definition}
\newtheorem{definition}[dummy]{Definition}
  \newtheorem{example}[dummy]{Example}
  
  \newtheorem{remark}[dummy]{Remark}

  \newtheorem*{acknowledgement}{Acknowledgement}

\newcommand
{\eqncount}{\setcounter{equation}{\value{dummy}}%
\addtocounter{dummy}{1}}

\newcommand{\cA}{\mathcal A}
\newcommand{\cB}{\mathcal B}

\newcommand{\cE}{\mathcal E}
\newcommand{\cF}{\mathcal F}
\newcommand{\cG}{\mathcal G}

\newcommand{\cL}{\mathcal L}
\newcommand{\cM}{\mathcal M}

\newcommand{\bZ}{\mathbb Z}
\newcommand{\bQ}{\mathbb Q}
\newcommand{\bC}{\mathbb C}
\newcommand{\bR}{\mathbb R}

\newcommand{\vv}{\, | \,}
\newcommand{\tensor}{\otimes}

\DeclareMathOperator{\Fix}{Fix}
\DeclareMathOperator{\Sign}{sign}
\DeclareMathOperator{\sign}{sign}

\DeclareMathOperator{\ad}{ad}

\DeclareMathOperator{\Ker}{Ker}
\DeclareMathOperator{\Coker}{Coker}

\newcommand{\bd}{\partial}
\newcommand{\CP}{\bC P}
\newcommand{\bCP}{\overline{\CP}}
\newcommand\cy[1]{\bZ/{#1}}
\DeclareMathOperator{\Sharp}{\#_1^n}
\newcommand\ra{\rangle}
\newcommand\la{\langle}
\newcommand{\bMX}{\overline{\cM(X)}}
\DeclareMathOperator{\lk}{lk}
\newcommand{\xa}{\hphantom{-}1}
\newcommand{\xb}{\hphantom{-}0}
\newcommand{\xc}{\hphantom{-}\vdots}

\begin{document}

\title{Cyclic Group Actions on Contractible $4$-Manifolds}

\author{Nima Anvari and Ian Hambleton}
\address{Department of Mathematics \& Statistics
 \newline\indent
McMaster University
 \newline\indent
Hamilton, ON  L8S 4K1, Canada}
\email{hambleton{@}mcmaster.ca}
\address{Department of Mathematics \& Statistics
 \newline\indent
McMaster University
 \newline\indent
Hamilton, ON  L8S 4K1, Canada}
\email{anvarin@math.mcmaster.ca}
\thanks{Research partially supported by NSERC Discovery Grant A4000.}
\begin{abstract}
There are known infinite families of Brieskorn homology $3$-spheres which can be realized as boundaries of smooth contractible $4$-manifolds. 
In this paper we show that smooth free periodic actions on these Brieskorn spheres do not extend smoothly over a contractible $4$-manifold.
We  give a new infinite family of examples in which the actions extend locally linearly but not smoothly. 
\end{abstract}

\maketitle
 

\section{Introduction}
\label{sect:intro}
The Brieskorn homology spheres $\Sigma(a,b,c)$ provide important examples of
Seifert fibered $3$-manifolds \cite{Orlik}, and have been extensively studied as test cases for questions about smooth $4$-manifolds and gauge theory invariants (see Anvari \cite{Anvari:2014},  Lawson \cite{Lawson87}, Fintushel and Stern \cite{FS90, FS91}, Saveliev \cite{Saveliev:2002}). In this paper we answer a well-known question  (asked by Allan Edmonds at Oberwolfach in 1988) about extending smooth free cyclic group actions on $\Sigma(a,b,c)$ to certain smooth $4$-manifolds which they bound.

Kwasik and Lawson \cite{KL93} found an infinite family of Brieskorn homology $3$-spheres
which admit free $\bZ/p$-actions and bound smooth contractible $4$-manifolds $W$, such that the actions extend locally linearly with one fixed point in $W$,  but no such extended action exists smoothly. Their examples come from the list of Casson and Harer \cite{CH81} of Brieskorn homology $3$-spheres which bound smooth contractible $4$-manifolds:
\begin{align*}
&\Sigma(r,rs-1,rs+1) \quad \text{$r$ even, $s$ odd} \\
&\Sigma(r,rs \pm 1,rs \pm 2) \quad \text{$r$ odd, $s$ arbitrary}.
\end{align*}
Necessary and sufficient conditions for a locally linear extension of  a free action on an integral homology three sphere to its bounding contractible $4$-manifold are contained in the work of Edmonds \cite{E87}. To show non-smoothability, Kawsik and Lawson apply the gauge theoretic results of Fintushel and Stern \cite{FS85} in the orbifold setting. 

In this paper we demonstrate a new technique to detect non-smoothability of these actions and apply it to obtain a complete answer:

\begin{thma} 
A free cyclic group 
action on a Brieskorn homology $3$-sphere $\Sigma(a,b,c)$ does not extend to  a smooth action on any contractible smooth $4$-manifold $W$ that it bounds.
\end{thma}
\begin{remark} By P.~A.~Smith theory,
any smooth or locally linear extension of a free cyclic action  
on 
$\Sigma(a,b,c)$ to a contractible manifold $W$ must have exactly one fixed point. 
\end{remark}

Recall that the Brieskorn homology spheres for $a,b,c$ pairwise relatively prime can be realized as the link of a complex surface singularity:

\begin{equation*}
\Sigma(a,b,c)=\{ (x,y,z) \in \bC^3 \vv x^a+y^b+z^c=0 \} \cap\, S^5
\end{equation*}
with its induced orientation. As a Seifert fibered homology sphere it admits a smooth fixed-point free circle action with three orbits of finite isotropy (see \cite{Orlik}). 

The action of $\pi = \bZ/p \subset S^1$ contained in the circle action will be free if and only if $p$ is relatively prime to $a,b,c$. This action is referred to as the \emph{standard $\pi$-action} on $\Sigma(a,b,c)$. Luft and Sjerve
\cite[Prop.~4.3]{LS92} showed that any smooth free cyclic group action on $\Sigma(a,b,c)$ is conjugate to a standard action.




We give new infinite family of examples admitting  locally linear extensions to a contractible $W$. The examples are contained in the second  of the infinite families found by Stern \cite{S78}:
\begin{align*}
&\Sigma(r,rs \pm 1,2r(rs \pm 1)+rs-(\pm 1)) \quad \text{$r$ even, $s$ odd} \\
&\Sigma(r,rs \pm 1,2r(rs \pm 1)+rs \pm 2) \quad \text{$r$ odd, $s$ arbitrary}\\
&\Sigma(r,rs \pm 2,2r(rs \pm 2)+rs \pm 1) \quad \text{$r$ odd, $s$ arbitrary}.
\end{align*}
where we take $s= kp$ for any positive integer $k$.

\begin{thmb}
Let $r$ be odd, and let $p$ be an integer relatively prime to $2r(r+1)$. Then for each positive integer $k$, the standard free action of $\pi=\bZ/p$ on $\Sigma(r,rkp\pm 1,2r(rkp \pm 1)+rkp \pm 2)$ extends to a locally linear $\pi$-action on a smooth contractible $4$-manifold $W$ that it bounds.
\end{thmb}

We point out the following application to equivariant embeddings of $\Sigma(a,b,c)$ in smooth homotopy $4$-spheres with semifree $\pi$-actions.

\begin{thmc}\label{cor:sphere}
There are infinite families of Brieskorn homology spheres $\Sigma(a,b,c)$ 
such that the standard free actions of $\pi=\bZ/p$Œ
 embed  equivariantly in homotopy $4$-spheres with locally linear $\pi$-actions. No such smooth  equivariant embeddings of $\Sigma(a,b,c)$ exist into any smooth $\pi$-action on a homotopy $4$-sphere. 
\end{thmc}

Here is brief outline of the paper. The links of complex surface singularities that are integral homology three spheres are \emph{plumbed} homology spheres; that is, they can be realized as the boundaries of smooth $4$-manifolds obtained by plumbing disk bundles over $2$-spheres with an intersection matrix that is negative definite. Among these is the canonical negative definite resolution in that it admits no $(-1)$-blowdowns. To prove non-smoothability of locally linear extensions,  we extend the free action on a Brieskorn homology sphere
$\Sigma = \Sigma(a,b,c)$  to its canonical negative definite resolution 
$M(\Gamma)$
 by equivariant plumbing on the resolution graph. From this we form the closed, simply connected $4$-manifold 
 $$X=M(\Gamma) \cup_{\Sigma} (-W)$$
  which by Donaldson's Theorem A \cite{DO86} has intersection matrix that is diagonalizable over the integers. If the action on $W$  is smoothable,  then $X$ admits a smooth $\bZ/p$-action which equivariantly splits along a free action on $\Sigma(a,b,c)$. The idea is that the global orientation of the moduli space prevents the configuration of invariant and fixed $2$-spheres in $M(\Gamma)$ obtained from plumbing to embed equivariantly and smoothly in a connected sum of linear actions on complex projective spaces. We use equivariant Yang-Mills moduli spaces as developed in Hambleton-Lee \cite{HL92, HL95}. 

In the next section we collect results from equivariant gauge theory that we will need for the proof of Theorem A, and in Section \ref{sect: locally linear extensions} we prove that locally linear extensions exist for the infinite family in Theorem B.  We work out explicit examples for the infinite family $\Sigma(3,3s+1,6(3s+1)+3s+2)$. 

\begin{acknowledgement} The authors would like to thank the referee for many helpful comments and suggestions.
\end{acknowledgement}
\section{Equivariant Moduli Spaces}
\label{sect: equivariant moduli spaces}

Let $(\Sigma, \pi)$ denote a Brieskorn integral homology $3$-sphere $\Sigma=\Sigma(a,b,c)$,  together with a free action of $\pi=\bZ/p$ contained in the natural circle action of the Seifert fibration,  and suppose this action extends smoothly to a contractible $4$-manifold $W$. 

\medskip
\noindent
\textbf{Conventions}. In this section, the notation $\pi$ denotes a finite cyclic group of \emph{prime} order. We also write $\pi=\cy p$ to specify the order. All $\pi$-actions are smooth and orientation-preserving.

\medskip
Now consider $(M(\Gamma),\pi)$ to be the canonical negative definite resolution of $\Sigma(a,b,c)$ together with the smooth free $\pi$-action extending via equivariant plumbing on the graph $\Gamma$. Then $X=M(\Gamma) \cup_{\Sigma} (-W)$ denotes a simply connected, smooth negative definite $4$-manifold together with a homologically trivial $\bZ/p$-action. As mentioned in the introduction, our strategy will be to study the equivariant instanton moduli spaces to obtain a contradiction to the action of $\pi$ extending smoothly to $W$. We begin this section by collecting results about equivariant Yang-Mills moduli spaces needed to prove non-smoothability of the extension (see \cite{DK90} and \cite{HL92}).

\medskip
Let $P\rightarrow X$ denote a principal $SU(2)$-bundle over a closed, smooth and simply connected $4$-manifold $X$ whose intersection form is odd and negative definite. By results of Donaldson and Wall, it follows that $X$ is homotopy equivalent to a connected sum of copies of $ \bCP^2$ (see \cite{DK90}).  Suppose that  $\pi=\bZ/p$  acts smoothly on $X$ inducing the identity on homology. We fix a real analytic structure on $X$ compatible with the group action and a real analytic $\pi$-invariant metric, so the action is given by real analytic isometries. 

Let $\cA$ denote the space of $SU(2)$ connections and $\cB=\cA/\cG$ the space of connections modulo the gauge group $\cG$. 
Since $SU(2)$-bundles are classified by the second Chern class $c_2(P) \in H^4(X;\bZ)$ and since the $\pi$-action on $X$ preserves the orientation, there are lifts of the isometries $g\colon X \rightarrow X$ that are  generalized bundle maps $\hat g\colon P\rightarrow P$. Let $\cG(\pi)$ denote the group of all lifts, then there is an action of $\cG(\pi)$ on the space of connections $\cA$ which is well-defined modulo gauge and there is a short exact sequence 
\eqncount
\begin{equation}
1 \rightarrow \cG \rightarrow \cG(\pi) \rightarrow \pi \rightarrow 1
\end{equation}
we then get a well-defined $\pi$-action on $\cB$. The metric induces a decomposition of $2$-forms $\Omega^2(\ad P)=\Omega^2_+(\ad P) \oplus \Omega^2_-(\ad P)$. We are interested in the ``charge one'' bundle, with $c_2(P)=1$,  and the Yang-Mills moduli space defined by connections modulo gauge with anti-self-dual (ASD) curvature:
\eqncount\begin{equation}
\cM(X)=\{ [A] \in \cB(P) \vv F_A^+=0 \}
\end{equation} 
Since the curvature is gauge invariant there is a natural $\pi$-action on $\cM(X)$. The stabilizer $\cG_A(\pi)$ has compact isotropy subgroups; when $A$ is irreducible $\Gamma_A=\{\pm 1 \}$ and when $[A]$ is reducible, $\Gamma_A=U(1)$ and the associated complex vector bundle $E\rightarrow X$ splits as $L\oplus L^{-1}$ for a complex line bundle $L$ over $X$. The stabilizer $\cG_A(\pi)$ is an extension in the short exact sequence
\eqncount\begin{equation}
1\rightarrow \Gamma_A \rightarrow \cG_A(\pi)\rightarrow \pi_A \rightarrow 1
\end{equation}
where $\pi_A$ denotes the stabilizer of $[A] \in \cM(X)$. 
The local finite-dimensional model of the moduli space is given by a $\cG_A(\pi)$-equivariant Kuranishi map
\eqncount\begin{equation}
\phi_A\colon H^1_A \rightarrow H^2_A
\end{equation} 
where $H^1_A$ and $H^2_A$ are the cohomology group of the $\cG_A(\pi)$-equivariant fundamental elliptic complex 
\eqncount\begin{equation}
0 \rightarrow \Omega^0(X;\ad P) \xrightarrow{d_A} \Omega^1(X;\ad P) \xrightarrow{d^+_A} \Omega^2_+(X;\ad P)\rightarrow 0
\end{equation}
where 
$$D_A^+=d^{\ast}_A+d^+_A\colon  \Omega^1(X;\ad P) \to  \Omega^0(X;\ad P) \oplus \Omega^2_+(X;\ad P)$$
  is the linearization of the ASD equation. The  \emph{formal dimension}  of this moduli space $\cM(X)$ is given by the fomula
  \eqncount
\begin{equation} \dim H^1_A - \dim H^0_A - \dim H^2_A = 5
\end{equation}
computed at any ASD connection. 
When $H^2_A=0$, the origin is a regular value for $\phi_A$ and the infinitesimal deformations in $H^1_A$ can be integrated,  so that a neighborhood of such an irreducible  ASD connection $[A]$ in the equivariant moduli space $(\cM(X),\pi)$ is locally isomorphic to $(\phi^{-1}(0)/\Gamma_A,\pi)$ and gives $5$-dimensional manifold charts on the moduli space. 

However, in this equivariant setting it is known that there are obstructions to equivariant transversaliy:  for example,  the virtual representation $[H_A^1]-[H_A^2] \in RO(\pi)$ must be an actual representation. Moreover it may not be possible to make an equivariant perturbation of the ASD equations to get $H^2_A=0$.  Hambleton and Lee in \cite{HL92} used the notion of equivariant general position as developed by Bierstone \cite{Bie77} and applied it to the setting of Yang-Mills moduli spaces. The idea is to make generic equivariant perturbations chart by chart giving the moduli space the structure of a equivariant stratified space. Here we list the main properties of the instanton moduli space in our setting when $X$ is negative definite.

\begin{enumerate}
\setlength{\itemsep}{7pt}
\item The equivariant moduli space $(\cM(X),\pi)$ is a Whitney stratified space which inherits an effective $\pi$-action and has open manifold strata parametrized by isotropy subgroups $\cM^{\ast}_{(\pi^{\prime})}=\{ [A] \in \cM^{\ast}(X) \vv \text{$A$ has isotropy subgroup $\pi_A=\pi^{\prime}$} \}$. 

\item An irreducible connection $[A]\in \Fix(\cM^\ast(X), \pi)$ corresponds to an equivariant lift of the $\pi$-action on $X$ to a $\cG_A(\pi)$-bundle structure on $P$, and
 the connected components of the fixed set 
  in the moduli space 
 correspond to distinct equivalence classes of lifts \cite{BM93}. In this case, $\cG_A(\pi)$ is a (possibly non-split) extension of $\pi$ by $\{\pm 1\}$.
 Moreover, the dimension of the fixed set can be computed from the $\pi$-fixed set of the fundamental elliptic complex:
\begin{equation*}
0 \rightarrow \Omega^0(X;\ad P)^{\pi} \xrightarrow{d_A} \Omega^1(X;\ad P)^{\pi} \xrightarrow{d^+_A} \Omega^2_+(X;\ad P)^{\pi}\rightarrow 0
\end{equation*}
for a connection $[A]$ in $\Fix(\cM^{\ast}(X),\pi)$. A fixed stratum is non-empty if its formal dimension is positive. In particular, the \emph{free stratum} $\cM^{\ast}_{(e)}$ is a $5$-dimensional, smooth, noncompact manifold consisting of irreducible ASD connections.

\item The strata have topologically locally trivial equivariant cone bundle neighborhoods.

\item There is an ideal boundary in the moduli space leading to $\pi$-equivariant Uhlenbeck-Taubes compactification $(\bMX ,\pi)$ consisting of highly-concentrated ASD connections parametrized by a copy of $X$:
\begin{equation*}
\bMX =\cM(X) \cup X
\end{equation*}
where $\cM(X)$ has a $\pi$-equivariant collar neighborhood diffeomorphic to $X \times (0,\lambda)$ for small $\lambda$ with the product action being trivial on $(0,\lambda)$.

\item There are equivariantly transverse charts at each reducible connection; that is $H_A^2=0$ for each reducible connection $[A]$,  and there exists a $\pi$-invariant neighborhood which is equivariantly diffeomorphic to a cone over some linear action on complex projective space $\bCP^2$.
\end{enumerate}

By equivariant general position, the closures of singular strata of dimension $\geq 5$ are disjoint from the closure of the free stratum.  Moreover, there is a connected component of the free stratum containing the collar and the set of reducible connections.  The fixed sets that occur in $\cM_{(e)}(X)$ have  even codimension. 
Since the disjoint high-dimensional singular strata play no role in our arguments, \emph{from now on the notation $\bMX$ will just mean the closure of the free stratum}.

A final ingredient in the Yang-Mills setting is the map
$$\mu\colon H_2(X;\bZ) \to H^2(\cB^{\ast};\bZ)$$
defined in \cite[\S 5.2]{DK90}. If $\tau\colon X \to \bMX $ denotes the inclusion of the Taubes boundary, then $\tau^*(\mu(\alpha)) = PD(\alpha)$, for any class $\alpha \in H_2(X;\bZ)$. Furthermore, for the restriction of $\mu(\alpha)$ to the copy of $\CP^{\infty}$ which links the reducible connection $A$, we have
$$\mu(\alpha)\vv_{\lk [A]}=-\la c_1(L),\alpha\ra h$$
where $h \in H^2(\CP^\infty;\bZ)$ is the positive generator, and $L \to X$ is the
complex line bundle in the splitting $E= L\oplus L^{-1}$ induced by $A$ (see \cite[5.1.21]{DK90}).
\begin{remark} \label{rem:reducibles}
 Interchanging the roles of $L$ and $L^{-1}$ leaves the gauge equivalence class $[A]$  unchanged, but the identifications $\psi_{\pm }\colon \lk [A] \cong \CP^{\infty}$ differ by complex conjugation. Hence the right-hand side of this formula is independent of the ordering $L, L^{-1}$
 (see \cite[Remark 2.29]{DO86}). 
\end{remark}

The Yang-Mills moduli spaces inherit a canonical orientation from that of $X$. The top exterior power of the tangent space $T_{[A]}\cM=\Ker D_A^+$ can be identified with the determinant line bundle of the elliptic complex 
\eqncount\begin{equation}
\det D_A^+=\Lambda^{\text{max}}(\Ker D_A^+) \tensor_{\bR} \Lambda^{\text{max}}(\Coker D_A^+) 
\end{equation}  
when $H_A^2=0$ and in \cite{DK90} it is shown that the determinant line bundle $\Lambda(P)$ is independent of deformations of $A$ and extends to $\cB^{\ast}$. Moreover, $\Lambda(P)$ admits a canonical trivialization giving an orientation on the free stratum $\cM_{(e)}$ and inducing the orientation
$$\text{(inward pointing normal)} \times \text{(given orientation on $X$)}$$
on the end of the moduli space as a collar  on the equivariant Taubes embedding of $(X,\pi)$ in $(\cM(X),\pi)$ (see \cite[p.~426]{Don87}). The canonical orientation  of $\cM$ near a link of a reducible connection agrees with that $\CP^2$ (see \cite[Example 4.3]{Don87}).

In the equivariant setting, we will show that there is a preferred generator $h \in H^2(\CP^2;\bZ)$ at the link of each reducible. An action $(X, \pi)$, where $\pi=\cy p$, is \emph{oriented} by fixing a negative definite orientation on $X$,  and  a  $\pi$-equivariant $Spin^c$-structure on $X$ for $p=2$.

\begin{theorem}[{\cite{HL95}, \cite{HT04}}]\label{thm:connected}  Let  $(X, \pi)$ be an oriented action of $\pi=\cy p$ on $X$. The fixed set $\Fix(\bMX ,\pi)$ is path connected,  and inherits a preferred orientation from the $\pi$-action on the moduli space.
\end{theorem}

\begin{proof} The first statement follows from \cite[Theorem C]{HL95}, but for convenience we include an outline of the proof.
First we suppose that $\pi=\cy p$, for $p$ an \emph{odd} prime. In this case,   the $\pi$-action induces a complex structure on the fibres of any $\pi$-equivariant $SO(2)$ bundle. This implies, for example, that a 2-dimensional component $F \subset \Fix(X, \pi)$ inherits a preferred orientation from the given orientation on $X$, and the complex structure induced by the $\pi$-action on the normal bundle of $F $ in $X$.


 In \cite[Lemma 8]{HL95} it is shown that the $\bZ/p$-action on the moduli space, for an odd prime $p$, induces a preferred orientation on the fixed set. The idea is that for any $\pi$-fixed ASD connection $[A]$ there is a splitting of the fundamental elliptic complex $\Omega^{\ast}=(\Omega^{\ast})^{\pi}\oplus (\Omega^{\ast})^{\perp}$ and 
\eqncount\begin{equation}
\Lambda(P)=\Lambda((\Omega^{\ast})^{\pi})\tensor \Lambda((\Omega^{\ast})^{\perp}).
\end{equation}
Since the fixed set has even codimension and $p$ is odd, the action induces a complex structure on $\Lambda((\Omega^{\ast})^{\perp})$ and hence a preferred orientation. Together with the canonical orientation of the moduli space, this induces a preferred orientation on the fixed set of any connected component containing $[A]$. The path connectedness of $\Fix(\bMX ,\pi)$ is proved in \cite[Theorem 3.11]{HT04}, where one key step is to show that every 1-dimensional fixed set in $\cM(X)^*$ has at least one reducible connection as a limit point (see \cite[Lemma 17]{HL95}). A counting argument (see \cite[p.~729]{HL95}) now completes the proof.

For involutions ($\pi=\cy 2$), the basic ingredient is a generalization due to Ono \cite{Ono:1993}  of a result of Edmonds \cite{Edmonds:1981}, namely that a $\pi$-equivariant $Spin^c$ structure on $(X, \pi)$ induces a preferred orientation on each 2-dimensional component of $\Fix(X, \pi)$. 

Since our actions are homologically trivial, 
the existence of a $\pi$-equivariant $Spin^c$ structure on $(X,\pi)$ follows by combining \cite[5.2]{edmonds1989aspects} and \cite[Theorem 6.2]{HL92}. Finally, recall that the fixed set of a linear involution on $\bCP^2$ always contains a fixed $2$-sphere. The fixed $2$-spheres in the links of the reducible connections in the moduli space are part of the 3-dimensional strata in $\cM(X)^*$, and by \cite[Theorem 16]{HL95} the closure of each component of these strata must intersect the Taubes boundary in a fixed $2$-sphere in $\Fix(X, \pi)$. 

This means that the  fixed $2$-spheres in the links of the reducible connections all inherit a preferred orientation from the choice of a $\pi$-equivariant $Spin^c$ structure on $(X,\pi)$. It follows that any two reducibles are in the closure of exactly one  component  of the fixed set in $\cM(X)^*$. A counting argument again shows that $\Fix(\bMX ,\pi)$ is path connected.
%
\end{proof}
\begin{corollary}\label{cor:preferred}
 Let $\{[A_1], [A_2], \dots, [A_n]\}$ denote the set of reducible connections in $\cM(X)$. If $(X, \pi)$ is oriented, then there is a preferred choice of 
generator $h_i \in H^2(\CP^\infty;\bZ)$ in the link $\lk [A_i] \cong \CP^\infty$ of each reducible. Equivalently, there is a preferred choice of line bundles $\{L_1, L_2, \dots, L_n\}$ such that  $E= L_i \oplus L_i^{-1}$ is the splitting induced by $A_i$, for $1\leq i \leq n$.
\end{corollary}
\begin{proof}
An oriented action $(X, \pi)$ has a preferred orientation on $\Fix(\bMX, \pi)$, and hence a preferred orientation on the $\pi$-fixed or $\pi$-invariant $2$-spheres in the linear actions on 
$$\CP^2=\lk [A_i] \cap \cM(X) \subset \CP^\infty$$
 in the link 
 of each reducible. The Poincar\'e duals of these oriented 2-spheres provide the preferred generators $h_i \in  H^2(\CP^\infty;\bZ)$, for $1\leq i \leq n$. We have seen in Remark \ref{rem:reducibles} that the choice of generator $\pm h_i$ corresponds to a choice of line bundle $L_i^{\pm}$.
\end{proof}

%
%
The $\mu$-map also provides $\pi$-invariant strata in the moduli space, since our actions $(X, \pi)$ are homologically trivial. The following construction will be used in the next section.
\begin{lemma}\label{lem:line bundle} For any $\alpha \in H_{2}(X;\bZ)$, the
 class $\mu(\alpha) \in H^2(\cB^{\ast};\bZ)$ corresponds to a $\pi$-equivariant line bundle $\cL_\alpha \rightarrow \cB^{\ast}$. Moreover, there exists an equivariant section $s$ of $\cL_\alpha$ restricted to $\cM^*(X)$,  so that the zero set  $V_\alpha =s^{-1}(0)$ is in equivariant general position in the moduli space.
\end{lemma}

\begin{proof} For $\pi$ a finite cyclic group, equivariant line bundles $L$ over a space  $(Y,\pi)$ are classified by a cohomology class  $[L]\in H^2(Y\times_\pi E\pi;\bZ)$. The natural map $ H^2(Y\times_\pi E\pi;\bZ) \to H^2(Y;\bZ)$ sends $[L] \mapsto c_1(L)$, and a  spectral sequence calculation shows that this map is surjective. This shows that there exists a $\pi$-equivariant line bundle $\cL_\alpha \rightarrow \cB^{\ast}$ with $c_1(\cL_\alpha) = \mu(\alpha)$. We may now restrict this line bundle to $\cM^*(X)$ and perturb the zero section into equivariant general position by the method of \cite{HL92}. The perturbation may be chosen so that $V_\alpha$ also intersects the links of the reducible connections in equivariant general position. 
\end{proof}

For a class $\alpha  \in H_2(X;\bZ)$ represented by an invariant $2$-sphere $F \subset X$, the zero section $V_\alpha$ of $\cL_\alpha$ is a stratified codimension two cobordism whose intersection with the Taubes collar may be chosen to be $F=\tau(X) \cap\, V_\alpha$. The other boundary components  provide  surfaces  in the links of  the reducible connections, that are $\pi$-invariant under the linear actions on complex projective spaces.

  \section{Smooth actions on negative definite $4$-manifolds}
  \label{sec:two}
  
The equivariant moduli space provides an equivariant stratified cobordism that relates a smooth $\pi$-action on a negative definite $4$-manifold to
 an equivariant connected sum of linear actions on complex projective spaces. 

\begin{example}[Linear Models] The complex projective plane $\CP^2$ admits linear actions of any finite cyclic group $\pi = \cy m $, given in homogeneous coordinates by the formula 
\eqncount
\begin{equation}\label{eq:action}
 t \cdot [z_0: z_1: z_2] = [z_0: \zeta^az_1: \zeta^bz_2], 
 \end{equation}
 where $t \in \pi$ is a generator, $\zeta= e^{2\pi i/m}$ is a primitive root of unity, and $a$ and $b$ are integers such that the greatest common divisor $(a, b, m) = 1$. For these actions, $\pi$ induces the identity on homology, and the singular set always contains the three fixed points $x_1=[1:0:0]$, $x_2=[0:1:0]$, and $x_3=[0:0:1]$. In addition, the three projective lines through the points $x_i$ and $x_j$, for $i\neq j$, are smoothly embedded $\pi$-invariant or $\pi$-fixed $2$-spheres with various isotropy subgroups depending on the values of $a$ and $b$ (see \cite[\S 1]{HT04}).
  \end{example}

\begin{remark} 
 Let $(X, \pi)$ denote an orientation-preserving smooth action of a cyclic group $\pi = \cy m$ on a closed oriented $4$-manifold $X$. If $m$ is odd, then the $\pi$-action induces an orientation at a fixed point $x_0 \in F$ for the normal bundle to a smoothly embedded surface $F \subset X$. In this way, any connected $\pi$-invariant surface containing a fixed point inherits a preferred orientation. For $m=2$  we need an extra ingredient, namely the existence of a $\pi$-equivariant $Spin^c$ structure, as discussed in the proof of Theorem \ref{thm:connected}.
 
 For example, in the above linear actions on $\CP^2$, the orientation induced  by the action on the projective lines $\CP^1 \subset \CP^2 $, each representing a primitive generator of $H_2(\CP^2) = \bZ$, is the complex orientation.
 \end{remark}
 
 From now on, we will work with smooth actions of cyclic groups $\pi = \cy p$ of \emph{prime} order on \emph{negative} definite $4$-manifolds 
 $X \simeq \Sharp  \bCP^2$. If $p=2$ there exists a $\pi$-equivariant $Spin^c$-structure on $(X, \pi)$.
 
 \begin{definition} Let $(X, \pi)$ be an oriented action on $X \simeq \Sharp  \bCP^2$. The \emph{standard} orientation on 
a connected $\pi$-invariant surface containing a fixed point is the orientation induced by the action, if $p$ is an odd prime, or the orientation induced by the $\pi$-equivariant $Spin^c$-structure, if $p=2$.
\end{definition}


For an oriented action $(X, \pi)$  of $\pi=\cy p$ on $X \simeq \Sharp  \bCP^2$, Corollary \ref{cor:preferred} provides a preferred generator $h_i \in H^2(\CP^\infty;\bZ)$, for $1\leq i \leq n$,  at the link $\lk [A_i]$ of each reducible. This is used in the following important definition.

 \begin{definition}\label{def:standard}
 Let $(X, \pi)$ be a smooth,  oriented action on $X\simeq \Sharp  \bCP^2$. A  diagonal basis $\{e_1, e_2, \dots, e_n\}$ for $H_2(X;\bZ)$, with $e_i \cdot e_j = - \delta_{ij}$,  is called a \emph{standard} basis if 
 $$\mu(e_i)\vv_{\lk [A_i]}=-\la c_1(L_i),e_i\ra h_i  = h_i \in H^2(\CP^\infty;\bZ),$$
 for $1\leq i \leq n$, where $[A_i]$ denotes the reducible connection in $\bMX$ determined by $\{ \pm e_i\}$ and $h_i$ denotes the preferred generator.
A standard basis is unique up to re-ordering of the basis elements.
\end{definition}

We can construct examples by equivariant connected sums at fixed points.
  The building blocks use the smooth, oriented $\pi$-actions on $\bCP^2$, 
  given by the formula \eqref{eq:action}. Note that the induced standard orientation on the projective lines is opposite to the complex orientation. We will use the notation
 $\bCP^1 \subset \bCP^2$ for this oriented embedded surface.
 
 The linear models of smooth homologically trivial $\pi$-actions on a connected sum $X = \Sharp  \bCP^2$ are then obtained by a \emph{tree} of equivariant connected sums, where we connect linear actions on $\bCP^2$ at fixed points. In order to preserve orientation, the tangential rotation numbers at the attaching points must be of the form $(c,d)$ and $(c, -d)$. 

 The equivariant moduli space shows that every smooth $\pi$-action on  an odd negative definite $4$-manifold strongly resembles an equivariant connected sum of linear actions.

\begin{theorem}[{\cite[Theorem C]{HL95}}]
Let $(X, \pi)$ be a smooth cyclic group action on $X\simeq \Sharp  \bCP^2$ inducing the identity on homology. Then there exists an equivariant connected sum of linear actions on $\bCP^2$ with the same fixed point data and tangential isotropy representations.
\end{theorem}

Let $F$ denote a fixed $2$-sphere for  the $\pi$-action on an equivariant connected sum of linear actions on  $\Sharp \bCP^2$. We give $F$ the standard orientation, and then 
 it is clear that the homology class $[F]$ can be written as $\sum_i a_i e_i$ for $a_i \in \{0,1\}$ in the  diagonal basis $e_i$ represented as $\overline{\CP^1} \subset \overline{\CP^2}$. 
The same statement holds  for smooth  $\pi$-actions on $X \simeq \Sharp\bCP^2$.
If $(X, \pi)$ is a homologically trivial action, then the fixed set consists of a disjoint union of isolated points and smoothly embedded $2$-spheres (see \cite[Proposition 2.4]{edmonds1989aspects}).

\begin{theorem}[{\cite[Thm.~16]{HL95}}]\label{thm:fixed set}
Let $\pi=\bZ/p$, for $p$ a prime, and $(X, \pi)$ be an oriented, smooth,  homologically trivial action on a smooth $4$-manifold $X  \simeq \Sharp \bCP^2$.  Then the integral homology class for each standardly oriented fixed $2$-sphere $F \subset X$ is given by an expression:
$$
[F]=\sum_i a_ie_i
$$
where $\{e_i\}$ is a  standard diagonal basis and $a_i \in \{ 0,1 \}$. 
\end{theorem}

\begin{proof} Since $X \simeq \Sharp \bCP^2$, we have a standard diagonal basis
$\{e_1, \dots, e_n\}$ for the intersection form on $H_2(X;\bZ)$. We can express
 $[F] = \sum_i a_i e_i$, for some integers $a_i$. 
Let $\hat e_i = PD(e_i)$ be the Poincar\'e dual to $e_i$, so that $\la \hat e_i, e_j\ra = \la \hat e_i \cup\hat e_j, [X]\ra = -\delta_{ij}$. Let
 $L_i $ denote the corresponding line bundle over $X$, with $c_1(L_i) = \hat e_i$,  which provides the reduction $E=L\oplus L^{-1}$ and a reducible  ASD connection $[A_i]$ on $L_i$. 
 
 In the compactified equivariant moduli space $\bMX $, the fixed set of the $\pi$-action is path connected, by Theorem \ref{thm:connected}.
 It follows that the links of the reducible connections all inherit  the same standard orientation as $\bCP^2$.

 If $V$ denotes the $3$-dimensional $\pi$-fixed stratum which is the zero set in Bierstone general position for $\mu([F]) = \sum a_i \mu(e_i)$, then $V$ inherits a preferred orientation from the free stratum, and the induced orientation on each component $\bd V_i = V \cap\, \lk [A_i]$ depends only on its homology class.
 
  Since the fixed strata in the links arise from a linear $\pi$-action on complex projective space,  we see that $\bd V = F \cup \bigcup \bd V_i$, where each non-empty component $\bd V_i$ in the link $\lk [A_i]$  is a fixed $2$-sphere representing the homology class of $\CP^1\subset \CP^2$.
 We now evaluate
 \eqncount
 \begin{equation}\label{eqn:mu}
 0 = \la \mu(e_k), [\bd V]\ra =
  \la \mu(e_k), \tau_*[F]\ra + \sum \la \mu(e_k), [\bd V_i]\ra
  \end{equation}
  But $\la \mu(e_k), \tau_*[F]\ra = \la PD(e_k), [F]\ra = -a_k$, and 
  $$\la \mu(e_k), [\bd V_i]\ra
  = -\la c_1(L_k), e_k\ra \la h_k, [\bd V_i]\ra = \delta_{ik}.$$
   since $h_k$ is the positive generator. 
  It follows that the coefficients in $[F] = \sum a_i e_i$ all have values in $\{ 0, 1\}$.
  \end{proof}
  
We will now generalize the statement of Theorem \ref{thm:fixed set} to handle smoothly embedded $\pi$-invariant $2$-spheres. Note that such a $2$-sphere is either fixed by $\pi$ or contains exactly two $\pi$-fixed points. In either case, the standard orientation is defined.

\begin{theorem}\label{thm:invariant}
Let $\pi=\bZ/p$, for $p$ a prime, and $(X, \pi)$ be an oriented, smooth,  homologically trivial action on a smooth $4$-manifold $X  \simeq \Sharp \bCP^2$.    Let $F \subset X$ be a smoothly embedded $\pi$-invariant $2$-sphere
 with the standard orientation. Then the homology class  $[F] \in H_2(X; \bZ)$ is given by the formula
$$
[F]=\sum_i a_ie_i
$$
where $\{e_i\}$ is a standard diagonal basis 
and  each $a_i \geq 0$. 
\end{theorem}
\begin{remark}
 If $F$ does not have the standard orientation, then each $a_i \leq 0$.
\end{remark}
\begin{proof} Let $F$ be a smoothly embedded $\pi$-invariant $2$-sphere in the action $(X, \pi)$. We assume that $F$ is not $\pi$-fixed, hence it contains exactly two isolated fixed points $x_0, x_1 \in F$. Let $\alpha = [F] \in H_2(X;\bZ)$ and let $V \subset  \bMX^*$ be the zero set of an equivariant section (in Bierstone general position) of the line bundle $\cL_\alpha$ given by $\mu(\alpha) \in H^2(\cB^*;\bZ)$. We may assume that $V \cap\, X = F$ at the Taubes boundary, and that $\bd V_i :=V \cap\, \lk [A_i]$ is a $\pi$-invariant surface in a linear action $(\bCP^2, \pi)$ for each reducible connection $[A_i]$. 

Note that without additional information, we can only conclude that the $\pi$-invariant surfaces  $\bd V_i$ are smoothly embedded in $\bCP^2$ except possibly in small neighbourhoods around the fixed points, where the surfaces might contain cones over $\pi$-invariant knots in $(S^3, \pi)$. At these points the embeddings are only topological (and not locally flat).

However,  we observe that the compactification $\overline{V}$ contains two $1$-dimensional $\pi$-fixed strata of $\bMX^*$ joining each of the isolated fixed points on $F$ to reducible connections, and passing through isolated fixed points on two components, say on $\bd V_0$ and $\bd V_1$. By Bierstone general position, the intersections $V \cap \, \lk [A_i]$ are equivariantly transverse at these points (for $i = 0, 1$). Moreover, the fixed set 
$$Z:=\Fix(\bMX, \pi)\cap\, \overline{V}$$
 is a tree by \cite[Theorem 3.11]{HT04}.
Since each link $(\bCP^2, \pi)$ has at most three isolated fixed points, and there is a unique  path in $Z$ from $x_0$ to $x_1$ (up to homotopy),  it follows that $\Fix(\bd V_i, \pi)$  contains exactly two fixed points for each non-empty $\pi$-invariant surface $\bd V_i$. At the ``initial" component $\bd V_0$, that contains a fixed point connected to $x_0 \in F$, we also see that the standard orientation on $\bd V_0$ agrees with the complex ``positive" orientation on $\CP^1 \subset \CP^2$. Since $Z$ is connected, each non-empty component $\bd V_i$ also inherits the positive orientation. It follows that  $\la h_k, [\bd V_i]\ra  \geq 0$, and the same calculation given in \eqref{eqn:mu} completes the proof.
\end{proof}

 \begin{remark}\label{rm:orientation}
Note that 
if $F \subset X$ is $\pi$-invariant but not $\pi$-fixed, and there exists a  $\pi$-fixed $2$-sphere $S$, standardly oriented and with $[S]^2 = -1$,  then $F$ has the standard orientation if  $[F]\cdot [S] = -1$. 
\end{remark}

\section{Proof of Theorem A}
\label{sect: proof of theorem A}
The minimal negative definite resolution for a Brieskorn homology sphere is obtained from 
the dual resolution graph of the singularity whose link is the Brieskorn homology $3$-sphere $\Sigma(a_1,a_2,a_3)$ (see Saveliev \cite[Ex.~1.17]{Saveliev:2002}).   For these singularities, the graph is a tree with weight $\delta$ on the central node,  and weights on the branches given by a continued fraction decomposition $a_i/b_i=[t_{i1},t_{i2},...,t_{im_i}]$ of the Seifert invariants.
 These weights are uniquely determined by the condition $t_{ij} \leq -2-a_i < b_i <0$ and 
\eqncount\begin{equation}
a_1 a_2 a_3 b_i / a_i \equiv 1 \pmod {a_i}, 
\end{equation}
where $\delta$ satisfies 
\eqncount\begin{equation}
\delta=\dfrac{-1}{a_1a_2a_3}+\sum_{i=1}^3 \dfrac{b_i}{a_i} \leq -1.
\end{equation}
  Fintushel and Stern defined the $R$-invariant for Brieskorn homology spheres, which is   an odd integer $R(a_1,a_2,a_3) \geq -1$. Moveover, if $\Sigma(a_1,a_2,a_3)$ bounds a smooth contractible manifold, then $R(a_1,a_2,a_3) = -1$ (see \cite[Theorem 1.1]{FS85}).  Neumann and Zagier \cite{NZ85} gave the calculation 
 $$R(a_1,a_2,a_3)=-2\delta-3.$$
  This implies that the central node  in the resolution graph of $ \Sigma(a_1,a_2,a_3)$ has weight  $\delta= -1$.

Equivariant plumbing on the defining graph $\Gamma$ gives the minimal negative definite resolution $M(\Gamma)$, where each node in the graph  is represented by an embedded $2$-sphere with self-intersection number given by its weight.
The circle action on $\Sigma(a, b,c)$ which arises from its Seifert fibering structure extends over the plumbing 
(see  \cite{Orlik}, \cite{Neumann78}). By construction, 
the central node sphere is  fixed under the circle action.

By restricting this circle action to $\pi=\cy p$, for any integer $p$ relatively prime to $a, b, c$, we obtain a simply connected, smooth $4$-manifold  $M(\Gamma)$ with a smooth, homologically trivial $\pi$-action,  whose boundary is $\Sigma = \Sigma(a,b,c)$ with the standard free $\bZ/p$-action.

Suppose that the standard free action on $\Sigma(a,b,c)$  also extends smoothly over  another compact smooth $4$-manifold $W$ with $\bd W = \Sigma$. Then we  obtain a smooth, closed $4$-manifold
\eqncount
\begin{equation}\label{eqn:acyclic}
X=M(\Gamma) \cup_{\Sigma} (-W)
\end{equation}
 together with  smooth, homologically trivial $\pi$-action. If $W$  is \emph{acyclic}, meaning that $W$ has the integral homology of a point,  and $\pi_1(W)$ is the normal closure of the image of $\pi_1(\Sigma)$, then $X$ will  be closed, 
 \emph{simply connected}, smooth $4$-manifold with \emph{odd} negative definite intersection form. In other words, $X \simeq \Sharp \bCP^2$ where $n = b_2(M(\Gamma))$. To prove Theorem A, it is enough to consider actions of $\pi=\cy p$ with $p$ prime.

\begin{theorem}\label{thm:A}
Suppose $\Sigma(a,b,c)$ bounds a smooth acyclic $4$-manifold $W$, such that $\pi_1(W)$ is the normal closure of the image of $\pi_1(\Sigma(a,b,c))$. If $p$ is a prime with $p \nmid abc$, then a free action of $\pi=\bZ/p$ on $\Sigma(a,b,c)$ does not extend to a smooth action on $W$.
\end{theorem}
\begin{proof} 
We form the manifold $X=M(\Gamma) \cup_{\Sigma} (-W)$ from the given acyclic manifold $W$ and the plumbed manifold $M(\Gamma)$ as described in \eqref{eqn:acyclic}. We have $X \simeq \Sharp \bCP^2$ where $n = b_2(M(\Gamma))$.
 There is a basis for  $H_2(X;\bZ)$ represented by the nodal $2$-spheres in the plumbing construction. Since the plumbing is done equivariantly, we obtain
 a configuration of smoothly embedded
  $\pi$-fixed $2$-spheres and $\pi$-invariant $2$-spheres in $X$, with at least one $\pi$-fixed $2$-sphere $F_1$ of self-intersection $-1$ (namely the central node in the graph $\Gamma$). We fix an ordering on the other nodes so that $F_2$ and $F_3$ are adjacent to $F_1$.
  
  We give each of these $2$-spheres the complex orientation 
 and let 
$$Q_X \colon H_2(X;\bZ) \times H_2(X;\bZ) \to \bZ$$
 denote the intersection form of $X$,  expressed as a matrix with respect to the basis 
 $$\cF=\{[F_1], [F_2], \dots, [F_n]\}. $$
  In other words, 
  $Q_X$ is the plumbing matrix defined by the graph $\Gamma$, in which $[F_i]\cdot [F_j] = 1$, for $i \neq j$, whenever this intersection is non-zero.

  Let $\cE=\{e_1, e_2, \dots, e_n\}$ denote a standard diagonal basis  given by an (orientation-preserving) homotopy equivalence $X \simeq \Sharp \bCP^2$, and the orientation convention given in Definition \ref{def:standard}.
  Let $C$ denote the change of basis matrix (with respect to $\cE$ and $\cF$),  so that $C^tQ_XC=-I$ is in diagonal form with respect to the basis $\cE$. Then the columns of $C$ give the components of  each $e_i$ in terms of the basis $\cF$, and similarly the columns of  $C^{-1}$ give the expressions for  each $F_i$ in terms of the standard diagonal basis $\cE$. 
 
  Since $F_1$ is a fixed $2$-sphere with $[F_1]\cdot [F_1] = -1$, we may assume that $e_1 = \pm [F_1]$ in the diagonal basis $\cE$. Suppose first that $e_1 = [F_1]$.
  The inverse $C^{-1}$ then has the form 
\eqncount\begin{equation}
C^{-1} =
 \begin{pmatrix}
  1 & -1 & -1 & \ast & \cdots & \ast \\
  0 &  a_2 & b_2 & \ast & \cdots & \ast \\
  0&  a_3 & b_3 & \ast & \cdots & \ast \\
  \vdots & \vdots  & \vdots & \vdots & \vdots & \vdots  \\
  0 & a_n & b_n & \ast & \cdots & \ast
 \end{pmatrix}  
\end{equation}
where we have labelled the base node $F_1$ and two adjacent nodes $F_2$ and $F_3$, such that $[F_2]\cdot [F_1] = [F_3]\cdot [F_1] = 1$, but $[F_2]\cdot [F_3] = 0$. By construction, $F_2$ and $F_3$ are $\pi$-invariant (but not fixed) embedded $2$-spheres. This configuration always occurs in the plumbing graph for $M(\Gamma)$. By Remark \ref{rm:orientation}, the complex orientation for $F_2$ and $F_3$ in the plumbing is opposite to the standard orientation.

 It follows that $[F_2] = -e_1 + a_2 e_2  \dots $, and similarly that $[F_3] = -e_1 + b_2 e_2  \dots $. 
By Theorem \ref{thm:invariant} we can conclude that all the non-zero entries in the second and third column are actually \emph{negative}.  On the other hand, 
since 
$$0 = [F_2]\cdot [F_3] = - 1-  \sum_{i=2}^n a_ib_i$$
and each term  $a_i b_i \geq 0$, we have a contradiction. If $e_1 = -[F_1]$, then $F_2$ and $F_3$ have the standard orientation and all the non-zero entries in the second and third columns of $C^{-1}$ must be positive (by Theorem  \ref{thm:invariant}). We obtain a contradiction as before.
\end{proof}

\section{Locally Linear Extensions}
\label{sect: locally linear extensions}

In this section we briefly survey some results of Edmonds \cite{E87} and Kwasik-Lawson \cite{KL93}.  First it should be noted, by the work of Freedman \cite{Freedman:1982}, that every integral homology $3$-sphere $\Sigma$ bounds a topological contractible $4$-manifold $W$. That every free action on $\Sigma$ can be extended to a topological action on a topological contractible $4$-manifold was first noted by Ruberman and Kwasik-Vogel \cite{KV86}. The question of extending a free action of  $\pi=\bZ/p$ on $\Sigma$ to a locally linear action on a contractible $4$-manifold was studied by Edmonds \cite{E87} for $p$ a given prime, including the case of an involution $p=2$. This work was generalized by Kwasik-Lawson \cite{KL93} to cover actions of any finite cyclic group.

\medskip
\noindent
\textbf{Conventions}. In this section, the notation $\pi$ denotes a finite cyclic group, not necessarily of prime order. We also write $\pi=\cy p$ to specify the order. All $\pi$-actions are orientation-preserving.

\medskip
The result for locally linear actions will involve additional spectral and torsion invariants.
 The equivariant eta invariant is the $g$-signature defect term for manifolds with boundary. Let $\partial W=\Sigma$ and $Q=\Sigma/\pi$, then the relation between the rho invariants $\rho(Q, \gamma)$ of the orbit space and the equivariant eta invariant is given by
\eqncount\begin{equation}
\eta_t(\Sigma)=\sum_{\gamma} \rho(Q, \gamma) \overline{\chi}_{\gamma}(t), \quad \text{for\ } t \in \pi,\  t \neq 1, 
\end{equation} 
where the sum contains values $\chi_\gamma(t)$ of the characters  of  the irreducible representations $\gamma$ of $\pi=\bZ/p$.
 There is also a Fourier transform formula \cite[2.8]{APS2} expressing rho invariants in terms of the equivariant eta invariant $\eta_t$:
\eqncount\begin{equation}
\rho(Q, \gamma)=\dfrac{1}{p} \sum_{t\neq 1}\eta_t(\Sigma)(\chi_{\gamma}(t)-\dim(\gamma)).
\end{equation}
As an example that we will use later, the rho invariants of classical lens spaces are given in terms of the representions $\gamma_\ell(t) = e^{2\ell\pi i/p}$: 
\eqncount\begin{equation}
\rho(L(p;r,s),\gamma_\ell)=\dfrac{4}{p}\sum_{k=1}^{p-1}\cot(\dfrac{\pi k r}{p})\cot(\dfrac{\pi k s}{p})\sin^2(\dfrac{\pi k \ell}{p})
\end{equation}
which can be easily computed from the above formula using the equivariant eta invariant $\eta_t(S^3)$ of the $3$-sphere with the action extending to a disk with rotation number $(r,s)$.
\eqncount\begin{equation}
\eta_t(S^3)=\dfrac{(t^r+1)(t^s+1)}{(t^r-1)(t^s-1)}, \quad \text{for\ }  t \in \pi,  \  t \neq 1.
\end{equation}
We will also need the notion of Reidemeister torsions before we state the main result about locally linear extensions. This torsion invariant arises from an acyclic chain complex as follows. Give $Q$ a cell structure and let $\Sigma$ be given the induced cell structure from the regular covering. Then $C_{\ast}(\Sigma)$ is a chain complex of free $\bZ[\pi]$ modules. Using the natural homomorphisms 
$$\bZ[\pi]\rightarrow \bZ[\zeta] \rightarrow \bQ[\zeta]$$
 where $t \mapsto \zeta=e^{2\pi i/p}$, we see that the twisted homology of $C_{\ast}(\Sigma) \tensor \bQ[\zeta]$ is acyclic with torsion $\Delta(Q)$ in $\bQ[\zeta]^{\times}$. The Reidemeister torsion of the lens space $L(p;r,s)$ is $\Delta (L(p;r,s)) \sim (\zeta^r-1)(\zeta^s-1)$.

\begin{theorem}[{Edmonds \cite{E87}, Kwasik-Lawson \cite[p.~32]{KL93}}]
\mbox{}
\begin{enumerate}
\item  A free action of $\pi=\bZ/p$ on an integral homology $3$-sphere $\Sigma$ extends to a locally linear action on a contractible $4$-manifold $W$   with one fixed point  if and only if the quotient rational homology sphere $Q=\Sigma/\pi$ is $\bZ[\pi]$ $h$-cobordant to a classical lens space $L$.\\
\item  A rational homology sphere $Q=\Sigma/\pi$ is $\bZ[\pi]$ $h$-cobordant to classical lens space $L$ if and only if there is a $\bZ[\pi]$-homology equivalence $f\colon Q \rightarrow L$ under which their rho invariants are equal and the Reidemeister torsions satisfy $\Delta(Q) \sim u^2\Delta(L)$ where $u$ is the image of a unit in $\bZ[\pi]$.
\end{enumerate}
\end{theorem}

Recall that a $\bZ[\pi]$ $h$-cobordism $V$ between $Q$ and $L$ is one where $H_{\ast}(V,Q;\bZ[\pi])=0$ with local coefficients; equivalently, the $\bZ/p$-cover is an integral $h$-cobordism. To find a locally linear extension, one needs to find a lens space $L(p,q)$ for some integer $q \pmod p$ and a $\bZ[\pi]$-homology equivalence $f\colon Q \rightarrow L(p,q)$ satisfying the conditions above. To do this, start with the classifying map of the cover $Q=\Sigma/\pi$, so a map $f\colon Q \rightarrow B\pi$. By general position arguments we can take the image to be a $3$-dimensional lens space $L(p;r,s)$ and arrange so that the map is of degree one \cite{E87}, thus giving a $\bZ[\pi]$-homology equivalence $f\colon Q \rightarrow L(p;r,s)$. When $\Sigma$ is a Seifert fibered space the following theorem gives the constraint on the lens space:

\begin{theorem}[{Kwasik-Lawson \cite[p.~35]{KL93}}]
Let $Q$ be a Seifert fibered space with Seifert invariants $\{(a_i,b_i)\}$ with $\alpha \sum 
{b_i}/{a_i} 
=p$ where $\alpha$ is the product of the $a_i$. Then there is a degree one map $f\colon Q\rightarrow L(p;r,s)$ which is a $\bZ[\pi]$-homology equivalence if and only if $\alpha \equiv rs \pmod p$.
\end{theorem}

In the case when we have a simple homology equivalence the Reidemeister torsion condition is fulfilled.

\begin{theorem}[{Kwasik-Lawson \cite[p.~37]{KL93}}]
There is a simple $\bZ[\pi]$-homology equivalence between the rational homology sphere $Q=\Sigma(a,b,c)/\pi$
 and a lens space $ L(p;r,s)$, respecting the orientation and the preferred generators of $H_1(Q)$ and $H_1(L)$,   if and only if $\{a,b,c\}$ are congruent to $\{r,s,1\} \pmod p$ up to sign and  $abc \equiv rs \pmod p$.
\end{theorem}

We now use the above results to show an infinite family in the list of Stern \cite{S78} admits locally linear pseudo free extensions to a contractible $4$-manifold. First we will need the following 
\begin{lemma}\label{lem:signature}
For each positive integer $k$, each of the Brieskorn homology $3$-spheres $\Sigma(r,rs \pm 1,2r(rs \pm 1)+rs \pm 2)$ bounds an indefinite smooth $4$-manifold $X_0$ with signature equal to $-2$.
\end{lemma}
\begin{proof}
We can realize $\Sigma(r,rs \pm 1,2r(rs \pm 1)+rs \pm 2)$ as the boundary of the following plumbed indefinite $4$-manifold $X_0$ (see Fickle \cite{Fickle}):
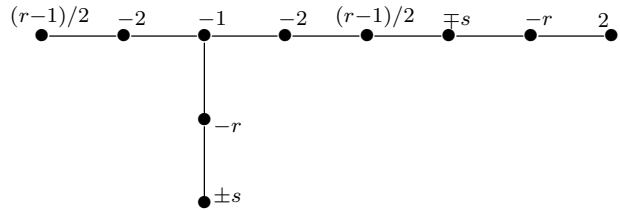
\begin{figure}[ht!]
\begin{displaymath}
\xymatrix{
  & *{\bullet} \ar@{-}[r]^<{(r-1)/2} & *{\bullet} \ar@{-}[r]^<{-2} & *{\bullet} \ar@{-}[d] \ar@{-}[r]^<{-1} & *{\bullet} \ar@{-}[r]^<{-2} & *{\bullet} \ar@{-}[r]^<{(r-1)/2} & *{\bullet} \ar@{-}[r]^<{\mp s} & *{\bullet} \ar@{-}[r]^<{-r} ^>{2} & *{\bullet} \\
                           & &             & *{\bullet}  \ar@{-}[d]^>{ \pm s} ^<{-r} &&&& \\
                           & &            & *{\bullet}  &&&& \\
}
\end{displaymath}
\caption{The boundary of this plumbed $4$-manifold is the homology $3$-sphere $\Sigma(r,rs \pm 1,2r(rs \pm 1)+rs \pm 2)$.}
\label{indefinite plumbing graph}
\end{figure}
The signature is determined via  an algorithm which amounts to a graph version of the Gaussian diagonalization process over the rationals (see \cite[p.~153]{EN85}).
\end{proof}

\begin{theorem}\label{thm:B}
Let $r$ be odd, and let $p$ be an integer relatively prime to $2r(r+1)$. Then for each positive integer $k$, the standard free action of $\pi=\bZ/p$  on $\Sigma(r,rkp \pm 1,2r(rkp \pm 1)+rkp \pm 2)$ extends to locally linear action on a smooth contractible $4$-manifold $W$ with a single fixed point of rotation number $(r,2r+2)$.   
\end{theorem}
\begin{proof}
There is a simple $\bZ[\pi]$-homology equivalence from the quotient
$$Q=\Sigma(r,rkp \pm 1,2r(rkp \pm 1)+rkp \pm 2)/\pi$$
 to the classical lens space $L(p;r,2r+2)$. We need to show that these have equivalent rho invariants and we do this by showing that their equivariant eta invariants are equal. 
Equivariant plumbing (see Fintushel \cite[\S 4]{Fintushel:1977}) on the graph in Figure \ref{indefinite plumbing graph} simplifies the computation: we will see  that it produces cancelling pairs of rotation numbers . 

For each integer $a$, let $D^2(a)$ denote the unit disk in $\bC$ with $S^1$-action given by $z \mapsto z^a$.  Given relatively prime integers $a$ and $b$, we have a circle action  on $D^2(a) \times D^2(b)$ given by the formula
\eqncount
\begin{equation}\label{eq:plumbing}
z \cdot (re^{i\theta},se^{i\tau})=(re^{i(\theta+at)},se^{i(\tau+bt)}),
\end{equation} 
where $z = e^{it} \in S^1$. 
 Write $S^2=D^2_+ \cup D^2_-$ as the upper and lower hemispheres and consider the trivial $D^2$-bundle over each hemisphere.
The formula  in \eqref{eq:plumbing}
defines an $S^1$-action on the trivial bundle $D^2_+ \times D^2$, and similarly for the lower hemisphere with $a$ and $b$ replaced with $c$ and $d$. 
We  glue these trivial equivariant bundles together using 
the map 
$$F\colon\partial D^2_+ \times D^2 \rightarrow \partial D^2_{-} \times D^2$$
 defined by $F(e^{i\theta},se^{i\tau})=(e^{-i\theta},se^{i(-k\theta+\tau)}).$
We obtain an $S^1$-equivariant $D^2$-bundle $E_k$ over $S^2 = D^2_+(a) \cup D^2_-(-a)$ with Euler number $k$,
provided that
\eqncount\begin{equation}
\begin{pmatrix}
c\\
d\\
\end{pmatrix}
=
\begin{pmatrix}
-1 & 0 \\
-k & 1 \\
\end{pmatrix}
\begin{pmatrix}
a\\
b\\
\end{pmatrix}.
\end{equation}
To equivariantly plumb with another such $D^2$ bundle we identify over a trivialized hemisphere  by interchanging base and fibre coordinates.  

The extended $\bZ/p$-action is part of the circle action and is therefore isotopic to the identity (hence homologically trivial). We can thus identify the equivariant signature of the manifold $X_0$  with its usual signature: $\sign (X_0) = -2$ (see Lemma \ref{lem:signature}). The rotation numbers arising from equivariant plumbing on the graph in Figure \ref{indefinite plumbing graph} are 
$$(2,r),(2,r), (-1,2), (-1,2), (r,-2), (r,-2), (-1,r), (1,r), (r, 2r+2)$$
and one fixed $2$-sphere with self-intersection $-1$ with rotation number $1$ acting on the normal fiber. The Euler characteristic of the fixed set $\chi(\Fix(X_0))=11$ and signature equal to $-2$. After removing the cancelling pairs the $G$-signature theorem for manifolds with boundary simplifies to
\eqncount\begin{equation}
\eta_t(\Sigma)=-2 \big(\frac{t+1}{t-1}\big) \big(\frac{t^2+1}{t^2-1}\big)+\dfrac{4t}{(t-1)^2}-\Sign(X_0)+\big(\frac{t^r+1}{t^r-1}\big) \big(\frac{t^{2r+2}+1}{t^{2r+2}-1}\big).
\end{equation}
It is easy to check that the first three terms above cancel leaving the equivariant eta invariant of the classical lens space $L(p;r,2r+2)$ as was to be shown. 
\end{proof}

\section{An Infinite Family of Examples}
\label{sec:six}

In this section we give an infinite family of examples of non-smoothable locally linear extensions.  

\begin{example}
The Brieskorn homology $3$-sphere $\Sigma=\Sigma(3,16,113)$  bounds a smooth contractible $4$-manifold $W$, and admits free $\pi=\bZ/5$-action. It is part of the infinite family of the form $\Sigma(r,rs+1,2r(rs+1)+rs+2)$ given by Stern's examples with $r=3$ and $s=5$. It follows from Theorem B that the standard $\bZ/5$-action on $\Sigma(3,16,113)$ extends to a locally linear action on $W$ with one fixed point whose rotation data is $(3,3)$. However,  Theorem A shows that there is no such smooth action. It follows that $\Sigma(3,16,113)$ admits a $\bZ/5$-equivariant embedding into a homotopy $4$-sphere with a locally linear $\cy 5$-action.

\begin{figure}[ht!]
\centering
\includegraphics[scale=.4]{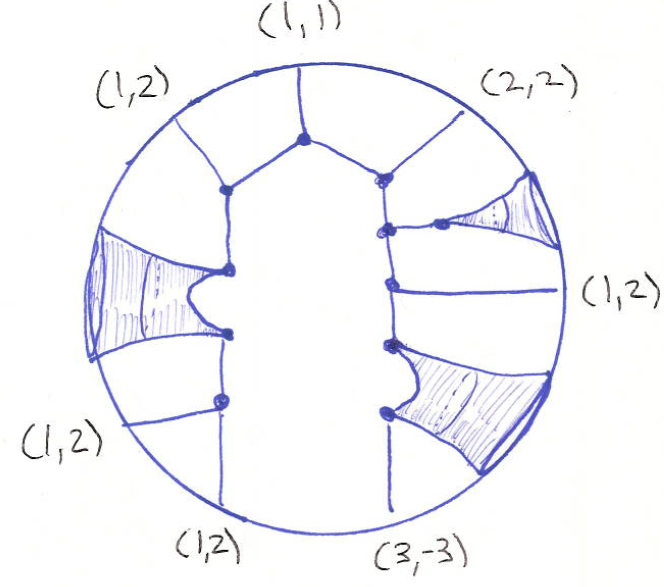}            
\caption[Moduli Space]{The fixed set pattern in the moduli space $(\cM(X),\pi)$ for $\Sigma(3,16,113)$.
 Each vertex in the interior is a reducible connection whose link is a complex projective space with a linear $\pi$-action. The isotropy representations then resemble that of an equivariant connected sum of linear actions on complex projective spaces.}\label{fig:two}
\end{figure}

The associated negative definite smooth $4$-manifold $M(\Gamma)$ has signature $-11$. 
\begin{figure}[ht!]
\begin{displaymath}
\Gamma = \vcenter{\xymatrix{
  & *{\bullet} \ar@{-}[r]^<{-3} & *{\bullet} \ar@{-}[d] \ar@{-}[r]^<{-1} & *{\bullet} \ar@{-}[r]^<{-4} & *{\bullet} \ar@{-}[r]^<{-2} & *{\bullet} \ar@{-}[r]^<{-2} & *{\bullet} \ar@{-}[r]^<{-2} ^>{-2} & *{\bullet} \\
                           &              & *{\bullet}  \ar@{-}[d]^>{-6} ^<{-3} &&&& \\
                           &              & *{\bullet}  \ar@{-}[d]^>{-4} &&&& \\
                           &              & *{\bullet}  \ar@{-}[d]^>{-2} &&&& \\
                           &              & *{\bullet}
}}
\end{displaymath}
\caption{The canonical negative definite plumbing diagram for $\Sigma(3,16,113).$}
\end{figure}
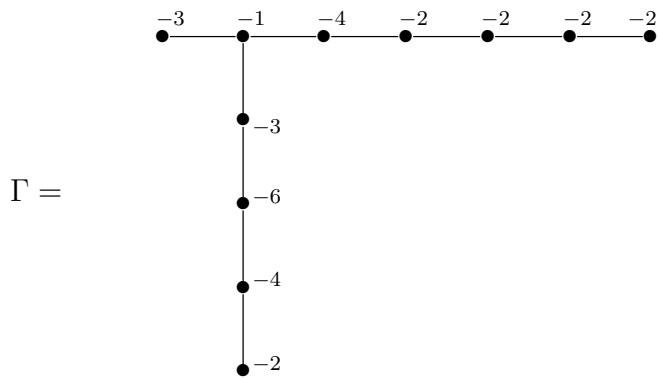
Equivariant plumbing beginning with the central vertex produces $6$ fixed points with rotation data $\{ (1,1),(1,2),(1,2),(1,2),(1,2),(2,2)\}$ and $3$ fixed $2$-spheres $F_1,F_2,F_3$, two of which represent homology classes of self-intersection $-2$ with normal rotation number $c_F=3$ and one representing a homology class (center vertex) of self-intersection $-1$ with normal rotation $c_F=1$. For the locally linear action on  $X =M(\Gamma) \cup_{\Sigma(3,16,113)} -W$, we have one additional fixed point with rotation data $(3,-3)$ coming from $-W$.

The intersection form $Q_X$ is given by
\eqncount\begin{equation}
Q_X=
\begin{pmatrix}
-1 & \xa & \xa & \xb   & \xb & \xb & \xa & \xb& \xb & \xb & \xb &\\
\xa & -3 & \xb & \xb   & \xb & \xb & \xb & \xb& \xb & \xb & \xb &\\
\xa & \xb & -3 & \xa   & \xb & \xb & \xb & \xb& \xb & \xb & \xb &\\
\xb & \xb & \xa & -6 & \xa & \xb & \xb & \xb& \xb & \xb & \xb &\\
\xb & \xb & \xb & \xa    & -4 & \xa & \xb & \xb& \xb & \xb & \xb &\\
\xb & \xb & \xb & \xb    & \xa & -2 & \xa & \xb& \xb & \xb & \xb &\\
\xa & \xb & \xb & \xb    & \xb & \xa & -4 & \xb& \xb & \xb & \xb &\\
\xb & \xb & \xb & \xb    & \xb & \xb & \xb & -2& \xa & \xb & \xb &\\
\xb & \xb & \xb & \xb    & \xb & \xb & \xb & \xa& -2 & \xa & \xb &\\
\xb & \xb & \xb & \xb    & \xb & \xb & \xb & \xb& \xa & -2 & \xa &\\
\xb & \xb & \xb & \xb    & \xb & \xb & \xb & \xb& \xb & \xa & -2 &\\
\end{pmatrix}
\end{equation}
and by Donaldson's Theorem A there exists an invertible integer matrix $C$ such that $C^t Q C=-I$, then the change of basis matrix $C^{-1}$ taking the basis in the plumbing diagram to a diagonal basis $\{e_i\}$ can be computed to be
\eqncount
\begin{equation}
C^{-1}=
\begin{pmatrix}
\xa & -1 & -1 & \xb  & \xb  & \xb   & -1  & \xb & \xb  & \xb  & \xb \\
\xb & -1  & \xa  & \xb & -1  & \xb   & \xb  & \xb  & \xb & \xb  & \xb \\
\xb & \xa  & \xb  & \xb  & -1 & \xb   & -1  & \xb  & \xb & \xb  & \xb \\
\xb & \xb  & \xa  & -1  & \xa  & \xb  & -1  & \xb  & \xb & \xb  & \xb \\
\xb & \xb  & \xb  & \xa  & \xb  & \xb   & \xb & \xb  & \xb & \xb  & -1 \\
\xb & \xb  & \xb  & \xb  & \xa  & -1   & \xb  & \xb  & \xb & \xb  & \xb \\
\xb & \xb  & \xb & \xb  & \xb  & -1   & \xb  & \xb  & \xb  & \xb & \xb \\
\xb & \xb  & \xb  & -1  & \xb  & \xb   & \xa  & -1 & \xb  & \xb  & \xb \\
\xb & \xb  & \xb  & -1  & \xb  & \xb   & \xb  & \xa  & -1  & \xb  & \xb \\
\xb & \xb  & \xb  & -1  & \xb  & \xb   & \xb  & \xb  & \xa  & -1  & \xb \\
\xb & \xb  & \xb  & -1  & \xb  & \xb   & \xb  & \xb  & \xb  & \xa  & -1 \\
\end{pmatrix}
\end{equation}
The fixed $2$-spheres are given in terms of a  diagonal basis as the first, sixth and tenth columns:
\begin{align*}
&F_1=e_1 \\
&F_6=-e_6-e_7 \\
&F_{10}=-e_{5}-e_{11}.
\end{align*}
The rest of the columns give expressions for the invariant $2$-spheres in terms of the diagonal basis. Now any other such matrix is obtained from $C$ by either permutations of the standard basis $\{e_i \}$ or a change of sign $(e_i \mapsto -e_i)$ since these are the automorphisms of the standard diagonal form. As we have seen, this information contradicts the existence of the general position equivariant moduli space  $\bMX$. Note that Figure \ref{fig:two} shows that this action is not ruled out just by the rotation numbers and the singular set in the moduli space.
\end{example}

\begin{example}
Before finding the general argument presented above, we worked out a particular infinite family of examples. Here is a way to simultaneously diagonalize all their intersection forms.
Let $(M_k,\pi)$ denote the canonical negative definite resolution of 
$$\Sigma_k=\Sigma(3, 3kp+1, 6(3kp+1)+3kp+2)$$ together with an action of  $\pi=\bZ/p$, for $p$ relatively prime to $6$,  extending the standard free $\pi$-action on $\Sigma_k$ via equivariant plumbing. 
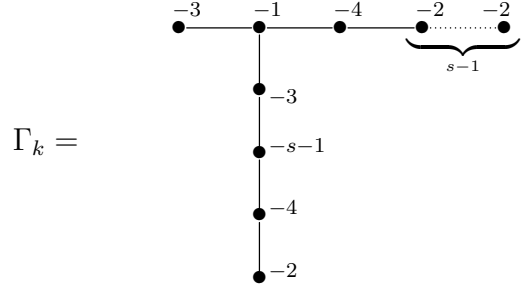
\begin{figure}[ht!]
\begin{displaymath}
\Gamma_k = \vcenter{\xymatrix@R-8pt{
  & *{\bullet} \ar@{-}[r]^<{-3} & *{\bullet} \ar@{-}[d] \ar@{-}[r]^<{-1} & *{\bullet} \ar@{-}[r]^<{-4} & *{\bullet} \ar@{.}[r]_{\underbrace{\hspace{1.5cm}}_{s-1}}^<{-2}^>{-2}  &  *{\bullet} \\
                           &              & *{\bullet}  \ar@{-}[d]^>{-s-1} ^<{-3} &&&& \\
                           &              & *{\bullet}  \ar@{-}[d]^>{-4} &&&& \\
                           &              & *{\bullet}  \ar@{-}[d]^>{-2} &&&& \\
                           &              & *{\bullet}
}}
\end{displaymath}
\caption{The canonical negative definite resolution plumbing diagram for $\Sigma(3,3s+1,6(3s+1)+3s+2)$, where $k =sp$.}
\end{figure}
If the action also extends to a smooth action on a contractible $4$-manifold $W$ then $X_k=M_k \cup -W$ is a simply connected, negative definite $4$-manifold with a smooth, homologically trivial $\pi$-action. The intersection form of $X_k$ is given by the  symmetric matrix indexed by $k$ (of size depending on $s= kp$):
$$
Q_{X_k}=
\begin{pmatrix}
-1 & \xa & \xa & \xb   & \xb & \xb & \xa & \xb& \xb & \xb & \xb & \xb&\cdots\\
\xa & -3 & \xb & \xb   & \xb & \xb & \xb & \xb& \xb & \xb & \xb & \xb&\cdots\\
\xa & \xb & -3 & \xa   & \xb & \xb & \xb & \xb& \xb & \xb & \xb & \xb&\cdots\\
\xb & \xb & \xa & -s-1 & \xa & \xb & \xb & \xb& \xb & \xb & \xb & \xb&\cdots\\
\xb & \xb & \xb & \xa    & -4 & \xa & \xb & \xb& \xb & \xb & \xb & \xb&\cdots\\
\xb & \xb & \xb & \xb    & \xa & -2 & \xa & \xb& \xb & \xb & \xb & \xb&\cdots\\
\xa & \xb & \xb & \xb    & \xb & \xa & -4 & \xb& \xb & \xb & \xb & \xb&\cdots\\
\xb & \xb & \xb & \xb    & \xb & \xb & \xb & -2& \xa & \xb & \xb & \xb&\cdots\\
\xb & \xb & \xb & \xb    & \xb & \xb & \xb & \xa& -2 & \xa & \xb & \xb&\cdots\\
\xb & \xb & \xb & \xb    & \xb & \xb & \xb & \xb& \xa & -2 & \xa & \xb&\cdots\\
\xb & \xb & \xb & \xb    & \xb & \xb & \xb & \xb& \xb & \xa & -2 & \xa&\cdots\\
\xb & \xb & \xb & \xb    & \xb & \xb & \xb & \xb& \xb &\xb  & \xa  &-2&\cdots\\
\xc&   \xc&   \xc&     \xc &  \xc&   \xc&   \xc&  \xc&  \xc &  \xc & \xc   &  \xc& \ddots
\end{pmatrix}
$$
We now claim that the matrix $C^{-1}$ in $C^t Q_{X_k} C=-I$ is given  in terms of a diagonal basis by the following expressions. Those which do not depend on the parameter are given by: 
\begin{align*}
&F_1=e_1, \quad F_2=-e_1-e_2+e_3, \quad F_3=-e_1+e_2+e_4,  \quad F_5=-e_2-e_3+e_4+e_6, \\
  &F_6=-e_6-e_7,\quad F_7=-e_1-e_3-e_4+e_8, \quad F_8=-e_8+e_9
\end{align*}
and the rest of the basis is obtained inductively by:
\begin{eqnarray*}
F_4=-e_4+e_5-e_8-e_9-e_{10} \cdots -e_{n}, \quad
F_n=-e_5-e_n  \quad F_{n-1}=-e_{n-1}+e_n
\end{eqnarray*}
where $n=6+s$, with $s\geq 3$. Once again, the point is that there is no consistent choice of sign in the expression of all the $[F_i]$, and moreover one cannot achieve such consistency by an automorphism of the standard form. Thus the actions in Theorem B for $r=3$ and $s=5$ do not extend smoothly.
\end{example}
\begin{proof}[The proof of Corollary C]
If $\Sigma=\Sigma(a,b,c)$ is a Brieskorn homology $3$-sphere which bounds a smooth contractible $4$-manifold $W$, then the manifold 
$N= W\cup_\Sigma (-W)$ is a smooth homotopy $4$-sphere in which $\Sigma(a,b,c)$ is a smoothly embedded submanifold. Now the examples of Theorem B provide a locally linear extension of the free $\pi=\cy p$-actions to $N$ with two isolated fixed points. \ Conversely, suppose that $(N, \pi)$ is a smooth $\pi$-action on a homotopy $4$-sphere. Then if $(\Sigma, \pi)$ embeds smoothly and equivariantly into $N$, it follows that the action on $N$ has two isolated fixed points, and that $N = W \cup_\Sigma W'$ is a smooth equivariant decomposition of $N$ as the union of compact $4$-manifolds $W$ and $W'$ with boundary $\Sigma$. By the van Kampen Theorem, the image of $\pi_1(\Sigma)$ normally generates $\pi_1(W)$, so we obtain a contradiction by Theorem \ref{thm:A}.
\end{proof}

\providecommand{\bysame}{\leavevmode\hbox to3em{\hrulefill}\thinspace}
\providecommand{\MR}{\relax\ifhmode\unskip\space\fi MR }
\providecommand{\MRhref}[2]{%
  \href{http://www.ams.org/mathscinet-getitem?mr=#1}{#2}
}
\providecommand{\href}[2]{#2}

\end{document}